\documentclass[10pt]{amsart}

\usepackage{amssymb,amsfonts,bm,amsmath}
\allowdisplaybreaks

\begin{document}

\theoremstyle{plain}
\newtheorem{thm}{Theorem}[section]
\newtheorem{theorem}[thm]{Theorem}
\newtheorem{lemma}[thm]{Lemma}
\newtheorem{corollary}[thm]{Corollary}
\newtheorem{proposition}[thm]{Proposition}
\newtheorem{conjecture}[thm]{Conjecture}
\theoremstyle{definition}
\newtheorem{construction}[thm]{Construction}
\newtheorem{notations}[thm]{Notations}
\newtheorem{question}[thm]{Question}
\newtheorem{problem}[thm]{Problem}
\newtheorem{remark}[thm]{Remark}
\newtheorem{remarks}[thm]{Remarks}
\newtheorem{definition}[thm]{Definition}
\newtheorem{claim}[thm]{Claim}
\newtheorem{assumption}[thm]{Assumption}
\newtheorem{assumptions}[thm]{Assumptions}
\newtheorem{properties}[thm]{Properties}
\newtheorem{example}[thm]{Example}
\newtheorem{comments}[thm]{Comments}
\newtheorem{blank}[thm]{}
\newtheorem{observation}[thm]{Observation}
\newtheorem{defn-thm}[thm]{Definition-Theorem}

\newcommand{\sM}{{\mathcal M}}


\title[A remark on Mirzakhani's asymptotic formulae]{A remark on Mirzakhani's asymptotic formulae}

\author{Kefeng Liu}
        \address{Center of Mathematical Sciences, Zhejiang University, Hangzhou, Zhejiang 310027, China;
                Department of Mathematics,University of California at Los Angeles,
                Los Angeles}
        \email{liu@math.ucla.edu, liu@cms.zju.edu.cn}

        \author{Hao Xu}
        \address{Center of Mathematical Sciences, Zhejiang University, Hangzhou, Zhejiang 310027, China;
               Department of Mathematics, Harvard University, Cambridge, MA 02138, USA}
        \email{haoxu@math.harvard.edu}

        \begin{abstract}
In this note, we answer a question of Mirzakhani on asymptotic
behavior of the one-point volume polynomial of moduli spaces of
curves. We also present some applications of Mirzakhani's asymptotic
formulae of Weil-Petersson volumes.
        \end{abstract}
    \maketitle

\section{Introduction} \label{secIn}

We will follow Mirzakhani's notation in \cite{Mir2}. For $\bold
d=(d_1,\cdots, d_n)$ with $d_i$ non-negative integers and $|\bold
d|=d_1+\cdots+d_n<3g-3+n$, let $d_0=3g-3+n-|\bold d|$ and define
\begin{equation} \label{eqM2}
[\tau_{d_1}\cdots\tau_{d_n}]_{g,n}=\frac{\prod^n_{i=1}(2d_i+1)!!2^{2|\bold
d|}(2\pi^2)^{d_0}}{d_{0}!}
\int_{\overline{\mathcal{M}}_{g,n}}\psi_1^{d_1}\cdots\psi_n^{d_n}\kappa_1^{d_0},
\end{equation}
where $\kappa_1=\omega/2\pi^2$ is the first Mumford class on
$\overline{\mathcal{M}}_{g,n}$ \cite{AC}. Note that
$V_{g,n}=[\tau_0,\cdots\tau_0]_{g,n}$ is the Weil-Peterson volume of
$\overline{\mathcal{M}}_{g,n}.$

Mirzakhani's volume polynomial is given by
$$V_{g,n}(2L)=\sum_{|\bold d|\leq3g-3+n}[\tau_{d_1}\cdots\tau_{d_n}]_{g,n}\frac{L_1^{2d_1}}{(2d_1+1)!}\cdots\frac{L_n^{2d_n}}{(2d_n+1)!}.$$

Let $S_{g,n}$ be an oriented surface of genus $g$ with $n$ boundary
components. Let $\mathcal M_{g,n}(L_1,\dots,L_n)$ be the moduli
space of hyperbolic structures on $S_{g,n}$ with geodesic boundary
components of length $L_1,\dots,L_n$. Then we know that the
Weil-Petersson volume ${\rm Vol}(\mathcal M_{g,n}(L_1,\dots,L_n))$
equals $V_{g,n}(L_1,\dots,L_n)$.

In particular, when $n=1$, Mirzakhani's volume polynomial can be
written as
$$V_{g}(2L)=\sum_{k=0}^{3g-2}\frac{a_{g,k}}{(2k+1)!}L^{2k},$$
where $a_{g,k}=[\tau_k]_{g,1}$ are rational multiples of powers of
$\pi$.
\begin{equation}
a_{g,k}=\frac{(2k+1)!!2^{3g-2+2k}\pi^{6g-4-2k}}{(3g-2-k)!}
\int_{\overline{\mathcal{M}}_{g,1}}\psi_1^k\kappa_1^{3g-2-k}.
\end{equation}

Let $\gamma$ be a separating simple closed curve on $S_{g}$ and
$S_{g}(\gamma)=S_{g_1,1}\times S_{g_2,1}$ the surface obtained by
cutting $S_g$ along $\gamma$. Then for any $L>0$, we have
\begin{equation}
{\rm Vol}(\mathcal M(S_g(\gamma),\ell_\gamma=L))=V_{g_1}(L)\cdot
V_{g_2}(L),
\end{equation}
where $\mathcal M(S_g(\gamma),\ell_\gamma=L)$ is the moduli space of
hyperbolic structures on $S_g(\gamma)$ with the length of $\gamma$
equal to $L$.

There are many works on the computation of Weil-Petersson volumes
(e.g. \cite{Fa, Gr, KMZ, MZ, Pe, ST, Wo, Zo}). In a recent paper
\cite{Mir2}, Mirzakhani proved some interesting estimates on the
asymptotics of Weil-Petersson volumes and found important
applications in the geometry of random hyperbolic surfaces. In
particular, Mirzakhani proved the following asymptotic relations of
the coefficients of the one-point volume polynomial.

\begin{theorem}{\rm\bf
(Mirzakhani \cite{Mir2})} For given $i\geq0.$
$$\lim_{g\rightarrow\infty}\frac{a_{g,i+1}}{a_{g,i}}=1, \qquad \lim_{g\rightarrow\infty}\frac{a_{g,3g-2}}{a_{g,0}}=0.$$
\end{theorem}

Mirzakhani asked what is the asymptotics of $a_{g,k}/a_{g, k+1}$ for
an arbitrary $k$ (which can grow with $g$). The following result
gives a partial answer to Mirzakhani's question.

\begin{theorem}\label{M}
For any given $k\geq0$, there is a large genus asymptotic expansion
\begin{equation}\label{eqM}
\frac{a_{g, 3g-2-k}}{g^k a_{g, 3g-2}}=\frac{\pi^{2k}}{5^k
k!}\left(1+\frac{b_{1,k}}{g}+\frac{b_{2,k}}{g^2}+\cdots\right).
\end{equation}
We have $b_{1,k}=\frac{1}{14}k^2-\frac{4}{7}k,\,\forall k\geq0$. In
fact, for any given $k\geq0$, the series in the bracket of
\eqref{eqM} is a rational function of $g$.
\end{theorem}

Theorem \ref{M} will be proved in Section \ref{secM}. Now we present
a numerical test of \eqref{eqM}. Denote by $Q_{k,g}$ the ratio of
the left-hand side and the truncated right-hand side of \eqref{eqM}.
\begin{equation}
Q_{k,g}=\frac{a_{g, 3g-2-k}}{g^k a_{g, 3g-2}}\cdot\frac{5^k
k!}{\pi^{2k}}\Big{/}\Big(1+\frac{b_{1,k}}{g}\Big).
\end{equation}
Then we can see from Table \ref{tb1} that $Q_{k,g}$ tends to $1$ as
$g$ goes to infinity.

\begin{table}[!htp]
\centering \caption{Values of $Q_{k,g}$ (keep 6 decimal
places)}\label{tb1}
\begin{tabular}{|c|c|c|c|c|c|}

\hline $k$ & $g=20$& $g=40$ & $g=60$ &
 $g=80$ & $g=100$

\\ \hline $1$ & $1.000438$ & $1.000106$ & $1.000047$ &
$1.000026$ & $1.000016$

\\ \hline $2$ & $1.001334$ & $1.000326$ & $1.000144$
& $1.000080$ & $1.000051$

\\ \hline $3$ & $1.002300$ & $1.000563$ & $1.000248$
& $1.000139$ & $1.000089$

\\ \hline $4$ & $1.003090$ & $1.000759$ & $1.000335$
& $1.000188$ & $1.000120$

\\\hline
\end{tabular}
\end{table}

\

\noindent{\bf Acknowledgements} The second author thanks Professor
M. Mirzakhani for very helpful communications.

\vskip 30pt
\section{Asymptotics of intersection numbers} \label{secM}
In this section, we use Witten's notation
\begin{equation}
\langle\tau_{d_1}\cdots\tau_{d_n}\kappa_{a_1}\cdots\kappa_{a_m}\rangle_g:=\int_{\overline{\mathcal{M}}_{g,n}}\psi_{1}^{d_{1}}\cdots\psi_{n}^{d_{n}}\kappa_{a_1}\cdots\kappa_{a_m}.
\end{equation}
For convenience, we denote the normalized tau function as
\begin{equation} \label{witten}
\langle\tau_{d_1}\cdots\tau_{d_n}\rangle_g^{\bf
w}:=\prod_{i=1}^n(2d_i+1)!!\langle\tau_{d_1}\cdots\tau_{d_n}\rangle_g.
\end{equation}

We have the following forms of the celebrated Witten-Kontsevich
theorem \cite{Wi, Ko}. The first one is called the DVV formula (see
\cite{DVV})
\begin{multline}\label{DVV}
(2d_1+1)!!\langle\tau_{d_1}\cdots\tau_{d_n}\rangle_g=\sum_{j=2}^n
\frac{(2d_1+2d_j-1)!!}{(2d_j-1)!!}\langle\tau_{d_2}\cdots
\tau_{d_{j}+d_1-1}\cdots\tau_{d_n}\rangle_g \\
+\frac{1}{2}\sum_{r+s=d_1-2}
(2r+1)!!(2s+1)!!\langle\tau_r\tau_s\tau_{d_2}\cdots\tau_{d_n}\rangle_{g-1}\\
+\frac{1}{2}\sum_{r+s=d_1-2} (2r+1)!!(2s+1)!!
\sum_{\{2,\cdots,n\}=I\coprod J}\langle\tau_r\prod_{i\in
I}\tau_{d_i}\rangle_{g'}\langle\tau_s\prod_{i\in
J}\tau_{d_i}\rangle_{g-g'}
\end{multline}
which is equivalent to the Virasoro constraint.

We also have the following recursive formula from integrating the
first KdV equation of the Witten-Kontsevich theorem (see Proposition
3.3 in \cite{LX1})
\begin{equation} \label{kdv}
(2g+n-1)\langle\tau_0\prod_{j=1}^n\tau_{d_j}\rangle_g=\frac{1}{12}\langle\tau_0^4\prod_{j=1}^n\tau_{d_j}\rangle_{g-1}
+\frac{1}{2}\sum_{\underline{n}=I\coprod
J}\langle\tau_0^2\prod_{i\in
I}\tau_{d_i}\rangle_{g'}\langle\tau_0^2\prod_{i\in
J}\tau_{d_i}\rangle_{g-g'}.
\end{equation}

\begin{definition}
The following generating function
$$F(x_1,\cdots,x_n)=\sum_{g=0}^{\infty}\sum_{\sum d_i=3g-3+n}\langle\tau_{d_1}\cdots\tau_{d_n}\rangle_g\prod_{i=1}^n x_i^{d_i}$$
is called the $n$-point function.
\end{definition}

In particular, we have Witten's one-point function
$$F(x)=\frac{1}{x^2}\exp\left(\frac{x^3}{24}\right),$$
which is equivalent to $\langle\tau_{3g-2}\rangle_g=1/(24^g g!)$.

The two-point function has a simple explicit form due to Dijkgraaf
(see \cite{Fa2})
$$F(x_1,x_2)=\frac{1}{x_1+x_2}\exp\left(\frac{x_1^3}{24}+\frac{x_2^3}{24}\right)\sum_{k=0}^{\infty}\frac{k!}{(2k+1)!}\left(\frac{1}{2}x_1x_2(x_1+x_2)\right)^k.$$
A general study of the $n$-point function can be found in
\cite{LX3}.

From Dijkgraaf's two-points function, it is not difficult to see
that
\begin{align*}
\lim_{g\rightarrow\infty}\frac{\langle\tau_k\tau_{3g-1-k}\rangle_k}{g^k\langle\tau_{3g-2}\rangle_g}
&=\lim_{g\rightarrow\infty}\frac{k!}{24^{g-k}(2k+1)!2^k(g-k)!}\cdot\frac{24^{g}\cdot
g!}{g^k}\\
&=\frac{k!24^{k}}{(2k+1)!2^k}\\
&=\frac{6^k}{(2k+1)!!}.
\end{align*}

In fact, we have the following more general result.

\begin{proposition} \label{M2}
For any fixed set $\bold d=(d_1,\dots, d_n)$ of non-negative
integers, the limit of the following quantity
\begin{equation}
C(d_1,\cdots,d_n;g)=\frac{\langle\tau_{d_1}\cdots\tau_{d_n}\tau_{3g-2+n-|\bold
d|}\rangle_g} {(6g)^{|\bold
d|}\langle\tau_{3g-2}\rangle_g}\prod^n_{i=1}(2d_i+1)!!
\end{equation}
exists and we have $\lim_{g\rightarrow\infty} C(d_1,\dots,d_n;g)=1$.
\end{proposition}
\begin{proof} We use induction on $|\bold d|$. When $d_1=\cdots=d_n=0$, it
is obviously true by the string equation.

From \eqref{kdv} and the string equation, we have that for any
$\bold k=(k_1,\dots,k_m)$ with $|\bold k|<|\bold d|$,
\begin{align}
\langle\prod_{i=1}^m\tau_{k_i}\tau_{3g-5+m-|\bold
d|}\rangle_{g-1}&\leq
\langle\tau_0^4\prod_{i=1}^m\tau_{k_i}\tau_{3g-1+m-|\bold d|}\rangle_{g-1}\nonumber\\
&\leq
12(2g+m)\langle\tau_0\prod_{i=1}^m\tau_{k_i}\tau_{3g-1+m-|\bold d|}\rangle_{g}\label{eqkdv2}\\
&=O\left(g\cdot \langle\prod_{i=1}^m\tau_{k_i}\tau_{3g-1+m-|\bold
d|}\rangle_{g}\right). \nonumber
\end{align}
Here $f_1(g)=O(f_2(g))$ means there exists a constant $C>0$
independent of $g$ such that $$f_1(g)\leq C f_2(g).$$ Note that the
last equation in \eqref{eqkdv2} is obtained by induction, since
$|\bold k|<|\bold d|$.

Let us expand $\langle\tau_{d_1}\cdots\tau_{d_n}\tau_{3g-2+n-|\bold
d|}\rangle_g$ using \eqref{DVV}. From \eqref{eqkdv2} and by
induction, we see that the second term in the right-hand side of
\eqref{DVV} has the estimate
\begin{equation}\label{eqkdv3}
\frac{1}{2}\sum_{r+s=d_1-2}
(2r+1)!!(2s+1)!!\langle\tau_r\tau_s\prod_{i=2}^n\tau_{d_i}\tau_{3g-2+n-|\bold
d|}\rangle_{g-1} =O\left(g^{|\bold d|-1}\right).
\end{equation}
Similarly, the third term in the right-hand side of \eqref{DVV} has
the estimate
\begin{equation}\label{eqkdv4}
\sum_{r+s=d_1-2} (2r+1)!!(2s+1)!! \sum_{\{2,\cdots,n\}=I\coprod
J}\langle\tau_r\prod_{i\in
I}\tau_{d_i}\rangle_{g'}\langle\tau_s\prod_{i\in
J}\tau_{d_i}\tau_{3g-2+n-|\bold d|}\rangle_{g-g'}=O\left(g^{|\bold
d|-2}\right).
\end{equation}

So by induction, we have
\begin{align}
\lim_{g\rightarrow\infty} C(d_1,\dots,d_n;g)=&\lim_{g\rightarrow\infty}\sum^n_{j=2}\frac{(2d_j+1) C(d_2,\cdots, d_{j}+d_1-1,\cdots d_n;g)} {6g}\nonumber\\
&+\lim_{g\rightarrow\infty}\frac{(2d_1+2(3g-2+n-|\bold
d|)-1)!!}{(2(3g-2+n-|\bold d|)-1)!!}\cdot\frac{C(d_2,\cdots,
d_n;g)}{(6g)^{d_1}} \label{eqkdv5}\\ =& 1. \nonumber
\end{align}
\end{proof}

\begin{corollary} \label{M9} We have the following large
genus asymptotic expansion
\begin{equation} \label{eqM15}
C(d_1,\dots,d_n; g)=1+\frac{C_1(d_1,\dots,d_n;
g)}{g}+\frac{C_2(d_1,\dots,d_n; g)}{g^2}+\cdots,
\end{equation}
where the coefficients $C_j(d_1,\dots,d_n; g)$ are determined
recursively by induction on $|\bold d|$,
\begin{multline} \label{eqM16}
C(d_1,\dots,d_n;g)=\frac{1}{6g}\sum_{j=2}^n
(2d_j+1)C(d_2,\dots,d_j+d_1-1,\dots,d_n;g)\\
+\frac{\prod_{j=1}^{d_1} (g+\frac{2n-2|\bold d|+2j-5}{6})}{g^{d_1}}
C(d_2,\dots,d_n;g)+\frac{(g-1)^{|\bold d|-2}}{3g^{|\bold d|-1}}\sum_{r+s=d_1-2} C(r,s,d_2,\dots,d_n;g-1)\\
+\sum_{r+s=d_1-2} \sum_{\{2,\cdots,n\}=I\coprod
J}24^{g'}6^{|J|+1-n-3g'}\langle\tau_r\prod_{i\in
I}\tau_{d_i}\rangle_{g'}^{\bold w}\\
\times\frac{(g-g')^{|J|+1-n+|\bold
d|-3g'}\prod_{j=1}^{g'}(g+1-j)}{g^{|\bold d|}}C(s,d_J;g-g'),
\end{multline}
where $d_J$ denote the set $\{d_i\}_{i\in J}$.

In fact, the expansion $C(d_1,\dots,d_n; g)$ has only finite nonzero
terms, i.e. $C_j(d_1,\dots,d_n; g)=0$ when $j$ is large enough.
\end{corollary}

\begin{proof}
The recursive relation follows from the asymptotic expansions of
equations \eqref{eqkdv3}, \eqref{eqkdv4} and \eqref{eqkdv5}. The
last assertion will follow from Corollary \ref{M12}.
\end{proof}

\begin{remark}

When $n=0$ or $|\bold d|=0$, we have
\begin{equation}
C(\emptyset ;g)=C(0,\dots,0;g)=1.
\end{equation}
By the string and dilaton equations, we have
\begin{align}
&C(0,d_2,\dots,d_n;g)=\frac{1}{6g}\sum_{j=2}^n
(2d_j+1)C(d_2,\dots,d_j-1,\dots,d_n;g)+ C(d_2,\dots,d_n;g),\\
&C(1,d_2,\dots,d_n;g)=(1+\frac{n-2}{2g})C(d_2,\dots,d_n;g).
\end{align}
So we may assume $d_i\geq2, \forall i$ in $C(d_1,\dots,d_n;g)$.
\end{remark}

\begin{remark} In large $g$ expansion, we have
\begin{equation}
\frac{1}{(g-m)^k}=\left(\sum_{i=1}^\infty
\frac{m^{i-1}}{g^i}\right)^k
\end{equation}
for any given $m$.

When $d_1,\dots,d_2\geq2$, from \eqref{eqM16} we can deduce that
\begin{equation} \label{eqM17}
C_1(d_1,\dots,d_n;g)=-\frac{|\bold d|^2}{6}+\frac{(n-1)|\bold
d|}{3}+\frac{n^2}{12}-\frac{5n}{12}.
\end{equation}
In particular,
\begin{align*}
C_1(d_1;g)&=-\frac{d_1}{6}-\frac13,\\
C_1(d_1,d_2;g)&=-\frac{1}{6}(d_1+d_2)^2+\frac13(d_1+d_2)- \frac12.
\end{align*}
\end{remark}

\

For the full expansion of $C(d_1,\dots,d_n;g)$, let us look at some
examples
\begin{align*}
C(1;g)=C(1,1;g)&=1-\frac{1}{2g},\\
C(2;g)&=1-\frac1g+\frac{5}{12g^2},\\
C(3;g)&=1-\frac{11}{6g}+\frac{95}{72g^2}-\frac{35}{72g^3}, \\
C(2,2;g)&=1-\frac{11}{6g}+\frac{17}{12g^2}-\frac{7}{12g^3}.
\end{align*}

In fact, we will see in a moment that the expansion \eqref{eqM15} of
$C(d_1,\dots,d_n; g)$ is a polynomial in $1/g$. Let
\begin{equation}
P_{d_1,\dots,d_n}(g)=(6g)^{|\bold d|}C(d_1,\dots,d_n;g).
\end{equation}
The recursive formula \eqref{eqM16} in Corollary \ref{M9} becomes
\begin{multline} \label{eqM14}
P_{d_1,\dots,d_n}(g)=\sum_{j=2}^n
(2d_j+1)P_{d_2,\dots,d_j+d_1-1,\dots,d_n}(g)\\
+\prod_{j=1}^{d_1} (6g+2n-2|\bold d|+2j-5)
P_{d_2,\dots,d_n}(g)+12g\sum_{r+s=d_1-2} P_{r,s,d_2,\dots,d_n}(g-1)\\
+\sum_{r+s=d_1-2} \sum_{\{2,\cdots,n\}=I\coprod
J}24^{g'}\langle\tau_r\prod_{i\in I}\tau_{d_i}\rangle_{g'}^{\bold
w}\prod_{j=1}^{g'}(g+1-j) P_{s,d_J}(g-g'),
\end{multline}

\begin{corollary} \label{M12}
For any fixed set $\bold d=(d_1,\dots, d_n)$ of non-negative
integers,
$$P_{d_1,\dots,d_n}(g)=\frac{\langle\tau_{d_1}\cdots\tau_{d_n}\tau_{3g-2+n-|\bold d|}\rangle_g}
{\langle\tau_{3g-2}\rangle_g}\prod^n_{i=1}(2d_i+1)!!$$ is a
polynomial in $\mathbb Z[g]$ with highest-degree term $6^{|\bold d|}
g^{|\bold d|}$. These polynomials $P_{d_1,\dots,d_n}(g)$ are
determined uniquely by the recursive relation \eqref{eqM14} and
$P_{\emptyset}(g)=P_{0,\dots,0}(g)=1$.
\end{corollary}

\begin{proof}
 By Theorem 4.3(iv) and
Proposition 4.4 in \cite{LX4}, we have
$$24^{g'}g'!\langle\tau_r\prod_{i\in I}\tau_{d_i}\rangle_{g'}^{\bold w}\in\mathbb Z.$$
Since $g'!$ divides $\prod_{j=1}^{g'}(g+1-j)$, it is not difficult
to see that $P_{d_1,\dots,d_n}(g)$ are polynomials with integer
coefficients by induction using \eqref{eqM14}.
\end{proof}

We introduce some notation. Consider the semigroup $N^\infty$ of
sequences ${\bold m}=(m(1),m(2),\dots)$ where $m(i)$ are nonnegative
integers and $m(i)=0$ for sufficiently large $i$. We also use
$(1^{m(1)}2^{m(2)}\dots)$ to denote $\bold m$.

Let $\bold m, \bold{a_1,\dots,a_n} \in N^\infty$, $\bold
m=\sum_{i=1}^n \bold{a_i}$.
$$|\bold m|:=\sum_{i\geq 1}i m(i)\quad ||\bold m||:=\sum_{i\geq1}m(i) \quad \binom{\bold m}{\bold{a_1,\dots,a_n}}:=\prod_{i\geq1}\binom{m(i)}{a_1(i),\dots,a_n(i)}.$$

Let $\bold m\in N^\infty$, we denote a formal monomial of $\kappa$
classes by
$$\kappa(\bold m):=\prod_{i\geq1}\kappa_i^{m(i)}.$$

The following remarkable identity was proved in \cite{KMZ}.
\begin{equation} \label{eqkmz}
\langle\prod_{j=1}^n\tau_{d_j}\kappa(\bold
m)\rangle_g=\sum_{p=0}^{||\bold m||}\frac{(-1)^{||\bold
m||-p}}{p!}\sum_{\substack {\bold
m=\bold{m_1}+\cdots+\bold{m_p}\\\bold{m_i}\neq\bold 0}}\binom{\bold
m}{\bold
{m_1,\dots,m_p}}\langle\prod_{j=1}^n\tau_{d_j}\prod_{j=1}^p\tau_{|\bold{m_j}|+1}\rangle_g.
\end{equation}

\subsection*{Proof of Theorem \ref{M}} For any $k\geq 1$, by definition we have
\begin{equation}
\frac{a_{g,3g-2-k}}{g^k
a_{g,3g-2}}=\frac{(6g-3-2k)!!2^{6g-4-2k}(2\pi^2)^k\langle\tau_{3g-2-k}\kappa_1^k\rangle_g/k!}{g^k(6g-3)!!2^{6g-4}\langle\tau_{3g-2}\rangle_g}.
\end{equation}
Using \eqref{eqkmz} to expand
$\langle\tau_{3g-2-k}\kappa_1^k\rangle_g$ and taking limit as
$g\rightarrow\infty$, we get by Proposition \ref{M2}
\begin{align*}
\lim_{g\rightarrow\infty}\frac{a_{g,3g-2-k}}{g^k
a_{g,3g-2}}&=\lim_{g\rightarrow\infty}
\frac{(6g-3-2k)!!(2\pi^2)^k\langle\tau_{3g-2-k}\tau_2^k\rangle_g}{g^k(6g-3)!!2^{2k}k!\langle\tau_{3g-2}\rangle_g}\\
&=\frac{\pi^{2k}}{5^k k!}\lim_{g\rightarrow\infty}\frac{15^k\langle\tau_{3g-2-k}\tau_2^k\rangle_g}{(6g)^{2k}\langle\tau_{3g-2}\rangle_g}\\
&=\frac{\pi^{2k}}{5^k
k!}\lim_{g\rightarrow\infty}C(\underbrace{2,\dots,2}_k;g)\\
&=\frac{\pi^{2k}}{5^k
k!}.
\end{align*}
So we get the leading term in the right-hand side of \eqref{eqM}.

Now we compute the coefficient of $1/g$ in the asymptotic expansion
of $a_{g,3g-2-k}/(g^k a_{g,3g-2})$. We have
\begin{multline} \label{eqM18}
\frac{a_{g,3g-2-k}}{g^k
a_{g,3g-2}}=\frac{(6g-3-2k)!!\pi^{2k}\left(\langle\tau_{3g-2-k}\tau_2^k\rangle_g-\frac{k(k-1)}{2}\langle\tau_{3g-2-k}\tau_2^{k-2}\tau_3\rangle_g\right)}{g^k(6g-3)!!2^{k}k!
\langle\tau_{3g-2}\rangle_g}+O(1/g^2)\\
=\frac{\pi^{2k}}{5^k
k!}\left(\frac{(6g)^k}{\prod_{j=1}^k(6g-2j-1)}C(\underbrace{2,\dots,2}_k;g)\right.\\
\left.-\frac{15}{14}k(k-1)\cdot
\frac{(6g)^{k-1}}{\prod_{j=1}^k(6g-2j-1)}
C(\underbrace{2,\dots,2}_{k-2},3;g)\right)+O(1/g^2).
\end{multline}

By \eqref{eqM17}, we have
\begin{equation}
C_1(\underbrace{2,\dots,2}_k;g)=\frac{1}{12}k^2-\frac{13}{12}k.
\end{equation}
Substituting it into \eqref{eqM18}, the coefficient of $1/g$ in the
asymptotic expansion of $a_{g,3g-2-k}/(g^k a_{g,3g-2})$ equals
\begin{equation}
C_1(\underbrace{2,\dots,2}_k;g)+\sum_{j=1}^k
\frac{1+2j}{6}-\frac{15}{14}k(k-1)\times\frac{1}{6}=\frac{1}{14}k^2-\frac{4}{7}k.
\end{equation}
So we get the second term in the right-hand side of \eqref{eqM},
namely
\begin{equation}\label{eqM19}
\frac{a_{g,3g-2-k}}{g^k a_{g,3g-2}}=\frac{\pi^{2k}}{5^k
k!}\left(1+\Big(\frac{1}{14}k^2-\frac{4}{7}k\Big)\frac{1}{g}+O(1/g^2)\right).
\end{equation}
Since there are only finite number of terms in the right-hand side
of \eqref{eqkmz}, from the above proof it is not difficult to see
that for each $k\geq1$, the series in the bracket of \eqref{eqM19}
is a rational function of $g$. So we conclude the proof of Theorem
\ref{M}.

\begin{example}
When $k=1$, we have
\begin{align*}
\frac{a_{g,3g-3}}{g a_{g,3g-2}}&=\frac{\pi^2}{5}\cdot \frac{6g}{6g-3}C(2;g) \\
&=\frac{\pi^2}{5}\cdot
\frac{12g^2-12g+5}{6g(2g-1)}\\
&=\frac{\pi^2}{5}\left(1-\frac{1}{2g}+\sum_{j=2}^\infty
\frac{1}{3\cdot 2^{j-1}g^j} \right).
\end{align*}

When $k=2$, we have
\begin{align*}
\frac{a_{g,3g-4}}{g^2
a_{g,3g-2}}&=\frac{\pi^4}{50}\left(\frac{(6g)^2}{(6g-3)(6g-5)}C(2,2;g)-\frac{15}{7}\cdot\frac{6g}{(6g-3)(6g-5)}C(3;g)\right)
\\
&=\frac{\pi^4}{50}\cdot\frac{(g-1)(1008g^3-1200g^2+888g-175)}{84g^2(2g-1)(6g-5)}\\
&=\frac{\pi^4}{50}\left(1-\frac{6}{7g}+\frac{43}{84g^2}+\cdots\right).
\end{align*}
These equations can be verified in low genera using the following
data:
\begin{gather*}
a_{1,0}=\frac{\pi^2}{12},\quad a_{1,1}=\frac{1}{2},\quad
a_{2,0}=\frac{29\pi^8}{192},\quad a_{2,1}=\frac{169\pi^6}{120},\quad
a_{2,2}=\frac{139\pi^4}{12},\\
a_{2,3}=\frac{203\pi^2}{3},\quad a_{2,4}=210,\quad
a_{3,0}=\frac{9292841\pi^{14}}{4082400},\quad
a_{3,1}=\frac{8497697\pi^{12}}{388800},\\
a_{3,2}=\frac{8983379\pi^{10}}{45360},\quad
a_{3,3}=\frac{127189\pi^8}{81},\quad a_{3,4}=\frac{94418\pi^6}{9},\\
a_{3,5}= \frac{166364\pi^4}{3},\quad
a_{3,6}=\frac{616616\pi^2}{3},\quad a_{3,7}=400400.
\end{gather*}
\end{example}

\begin{corollary}
For any $\bold m=(m(1),m(2),\dots)\in N^\infty$, we have the
following limit equation involving higher degree $\kappa$ classes
\begin{equation}
\lim_{g\rightarrow\infty}\frac{\langle\prod_{i=1}^n\tau_{d_i}\tau_{3g-2+n-|\bold
d|-|\bold m|}\kappa(\bold m)\rangle_g}{(6g)^{|\bold d|+|\bold
m|+||\bold m||}\langle\tau_{3g-2}\rangle_g}=\frac{\bold m!}{||\bold
m||!\prod_{i=1}^n (2d_i+1)!!\prod_{j\geq1}((2j+3)!!)^{m(j)}}.
\end{equation}
\begin{proof}
This identity follows directly from Proposition \ref{M2} and
equation \eqref{eqkmz}.
\end{proof}

\end{corollary}

\vskip 30pt
\section{Asymptotics of Weil-Petersson volumes} \label{secZ}

The large genus asymptotics of Weil-Petersson volumes was
conjectured by Zograf based on his numerical experiments \cite{Zo}.
\begin{conjecture}{\rm\bf (Zograf)} For any fixed $n\geq 0$
$$V_{g,n} =(4\pi^2)^{2g+n-3} (2g-3+n)! \frac{1}{\sqrt{g \pi}} \left(1 + \frac{c_{n}}{g} + O\left(\frac{1}{g^{2}}\right)\right)$$
as $ g\rightarrow \infty$, where $c_n$ is a constant depending only
on $n$.
\end{conjecture}
Note that the asymptotic behavior of $V_{g,n}$ for fixed $g$ and
large $n$ has been determined by Manin and Zograf \cite{MZ}. Next We
recall Mirzakhani's work in \cite{Mir2}. We use the notation
introduced in Section \ref{secIn}. For $n\geq 0$, define
$$a_{n} =\zeta(2n) (1-2^{1-2n}).$$
We have the following properties of $a_n$.
\begin{lemma}\label{aa}{\rm\bf (Mirzakhani
\cite{Mir2})} $\{a_{n}\}_{n=1}^{\infty}$ is an increasing sequence.
Moreover we have $\lim_{n\rightarrow \infty} a_{n}=1,$ and
\begin{equation}\label{abound}
a_{n+1}-a_{n} \asymp 1/2^{2n}.
\end{equation}
Here $f_1(n) \asymp f_2(n)$ means
that there exists a constant $C>0$ independent of $n$ such that
$$\frac{1}{C} f_{2}(n) \leq f_1(n) \leq C f_{2}(n).$$
\end{lemma}

We have the following differential form of Mirzakhani's recursion
formula \cite{Mir1, MS} (see also \cite{Sa, LX1, LX2, EO}).
\begin{equation} \label{eqM3}
[ \tau_{d_{1}},\ldots,\tau_{d_{n}}]_{g,n}=8\left(
\sum_{j=2}^{n}\mathcal{A}^{j}_{{\bf d}} + \mathcal{B}_{{\bf d}}+
 \mathcal{C}_{{\bf d}}\right),
\end{equation}
where
\begin{equation}\label{A}
 \mathcal{A}^{j}_{{\bf d}}= \sum_{L=0}^{d_{0}} (2d_{j}+1) \; a_{L} [\tau_{d_{1}+d_{j}+L-1}, \prod_{i\not=1,j} \tau_{d_{i}}]_{g,n-1},\end{equation}

\begin{equation}\label{B}
 \mathcal{B}_{{\bf d}}= \;\sum_{L=0}^{d_{0}} \sum_{k_{1}+k_{2}=L+d_{1}-2} a_{L} [\tau_{k_{1}} \tau_{k_{2}} \prod_{i\not=1} \tau_{d_{i}}]_{g-1,n+1},
 \end{equation}
 and
\begin{equation}\label{C}
\mathcal{C}_{{\bf d}}= \sum_{I \amalg J=\{2,\ldots,n\}\atop 0\leq g'
\leq g} \sum_{L=0}^{d_{0}} \sum_{k_{1}+k_{2}=L+d_{1}-2} a_{L} \;
[\tau_{k_{1}} \prod_{i\in I } \tau_{d_{i}}]_{g',|I|+1} \times [
\tau_{k_{2}} \prod_{i\in J} \tau_{d_{i}}]_{g-g',|J|+1}.
\end{equation}

\begin{lemma}
 Given ${\bf d}=(d_{1},\ldots, d_{n})$ and $g,n\geq0$, the following recursive formulas hold
\begin{equation}\label{kdv2}
 [ \tau_{0} \tau_{1} \prod_{i=1}^{n} \tau_{d_{i}}]_{g,n+2} = [\tau_{0}^{4} \prod_{i=1}^{n} \tau_{d_{i}}]_{g-1,n+4}
 +6\sum_{g_{1}+g_{2}=g \atop \{1,\ldots, n\}=I\amalg J} [ \tau_{0}^{2} \prod_{i\in I } \tau_{d_{i}}]_{g_{1},|I|+2} [\tau_{0}^{2}\prod_{i\in J}
 \tau_{d_{i}}]_{g_{2},|J|+2},
\end{equation}
\begin{equation} \label{dilaton}
(2g-2+n) [ \prod_{i=1}^{n} \tau_{d_{i}}]_{g,n}= \frac{1}{2}
\sum_{L\geq0} (-1)^{L} (L+1) \frac{\pi^{2L}}{(2L+3)!} [ \tau_{L+1}
\prod_{i=1}^{n} \tau_{d_{i}}]_{g,n+1},
\end{equation}
\begin{equation}\label{string}
\sum_{j=1}^n(2d_j+1)[\tau_{d_j-1}\prod_{i\neq
j}\tau_{d_i}]_{g,n}=\sum_{L\geq0}\frac{(-\pi^{2})^L}{4(2L+1)!}[
\tau_{L} \prod_{i=1}^{n} \tau_{d_{i}}]_{g,n+1}.
\end{equation}
\end{lemma}

The above three equations in such forms were stated at Section 3 of
\cite{Mir2}. Mirzakhani proved the following remarkable asymptotic
formulae based on the data computed by Zograf \cite{Zo}.
\begin{theorem}\label{M4}{\rm\bf (Mirzakhani
\cite{Mir2})} Let $n\geq 0$. Then we have
\begin{equation} \label{eqM11}
\frac{V_{g,n+1}} {2g V_{g,n}}= 4\pi^{2}+ O(1/g)
\end{equation}
and
\begin{equation}
\frac{V_{g,n}} {V_{g-1,n+2}}=1+O(1/g).
\end{equation}
\end{theorem}

Following Mirzakhani's notation, denote
$$[{\bf x}]_{g,n}:= [\tau_{x_{1}}\ldots \tau_{x_{n}}]_{g,n},$$
where ${\bf x}=(x_{1},\ldots,x_{n}).$

\begin{lemma}\label{M5}{\rm\bf (Mirzakhani
\cite{Mir2})} In terms of the above notation, for ${\bf
x}=(x_{1},\ldots, x_{l}),$ and ${\bf y}=(y_{1},\ldots, y_{m})$, we
have
\begin{equation}\label{eqM8}
\sum_{g_{1}+g_{2}=g} [{\bf x}]_{g_{1},l} \times [{\bf
y}]_{g_{2},m}=o(V_{g,n-2}),
\end{equation}
where $n=l+m$.
\end{lemma}

The above lemma is a weaker form of Lemma 3.3 in \cite{Mir2}.

\begin{lemma} \label{M10} When $d_1>0$, we have
\begin{equation} \label{eqM4}
[\tau_{d_1}\cdots\tau_{d_n}]_{g,n}<
[\tau_{d_1-1}\tau_{d_2}\cdots\tau_{d_n}]_{g,n}.
\end{equation}

\end{lemma}

\begin{proof} We expand both sides of the inequalities using
\eqref{eqM3}. Since each term in $\mathcal{A}^{j}_{\bf d},
\mathcal{B}_{\bf d}, \mathcal{C}_{\bf d}$ is positive, by comparing
corresponding terms in the expansion, the inequality \eqref{eqM4}
follows from Lemma \ref{aa} that $\{a_{n}\}_{n=1}^{\infty}$ is a
strictly increasing sequence.
\end{proof}

\begin{corollary} \label{M3}
For any fixed set $\bold d=(d_1,\dots, d_n)$ of non-negative
integers, we have
\begin{equation} \label{eqM7}
[\tau_{d_1}\cdots\tau_{d_n}]_{g,n}\leq V_{g,n}.
\end{equation}
\end{corollary}

We can now prove the following Zograf's conjecture \cite{Zo} giving
large genus ratio of Weil-Peterson volumes and intersection numbers
involving $\psi$-classes. The proof is essentially due to Mirzakhani
\cite{Mir2}.
\begin{theorem} \label{Z}
For any fixed $n>0$ and a fixed set $\bold d=(d_1,\cdots,d_n)$ of
non-negative integers, we have
\begin{equation}\label{eqZ}
\lim_{g\rightarrow\infty}\frac{[\tau_{d_1}\cdots\tau_{d_n}]_{g,n}}{V_{g,n}}=1.
\end{equation}
\end{theorem}
\begin{proof} We use induction on $|\bold
d|$. We need only prove the following limit equation
\begin{equation} \label{eqM5}
\lim_{g\rightarrow\infty}\Bigg|\frac{[\tau_{d_1}\cdots\tau_{d_n}]_{g,n}}{[\tau_{d_1-1}\tau_{d_2}\cdots\tau_{d_n}]_{g,n}}-1\Bigg|=0
\end{equation}
By induction, we may assume
\begin{equation}
\lim_{g\rightarrow\infty}\frac{[\tau_{d_1-1}\tau_{d_2}\cdots\tau_{d_n}]_{g,n}}{V_{g,n}}=1.
\end{equation}

So in order to prove \eqref{eqM5}, we need only prove that
\begin{equation}\label{eqM6}
\lim_{g\rightarrow\infty}\Bigg|\frac{[\tau_{d_1-1}\tau_{d_2}\cdots\tau_{d_n}]_{g,n}-[\tau_{d_1}\cdots\tau_{d_n}]_{g,n}}{V_{g,n}}\Bigg|=0.
\end{equation}
By comparing each term in Mirzakhani's recursion formula
\eqref{eqM3} for $[\tau_{d_1-1}\tau_{d_2}\cdots\tau_{d_n}]_{g,n}$
and $[\tau_{d_1}\cdots\tau_{d_n}]_{g,n}$, this actually follows from
\eqref{eqM7}, \eqref{abound}, Theorem \ref{M4} and Lemma \ref{M5}.
The argument is similar to the proof of Theorem 3.5 in \cite{Mir2}.
We omit the details.
\end{proof}

\begin{remark}
We thank Mirzakhani \cite{Mir3} for pointing out that Zograf was
able to prove Theorem \ref{Z} using the method of \cite{MZ}.
\end{remark}

\begin{lemma}\label{M11}
When $3g+n-2>0$, we have
\begin{equation} \label{eqM13}
V_{g,n+1}\leq\frac{\pi^2}{6}[\tau_1\tau_0^n]_{g,n+1}.
\end{equation}
The equality holds only when $(g,n)=(0,3)$ or $(1,0)$.
\end{lemma}
\begin{proof} First note that the coefficients in \eqref{string}
$$\left\{\frac{\pi^{2L}}{4(2L+1)!}\right\}_{L\geq1}$$
is a decreasing sequence.

From Lemma \ref{M10}, we know $[ \tau_{L} \prod_{i=1}^{n}
\tau_{d_{i}}]_{g,n+1}$ is a decreasing sequence in $L$.

Taking all $d_i=0$ in \eqref{string}, the left-hand side becomes
$0$. Writing down the first two terms of the right-hand side, we get
$$\frac14 V_{g,n+1}-\frac{2\pi^2}{2^4\cdot
3}[\tau_1\tau_0^n]_{g,n+1}<0,$$ which is just \eqref{eqM13}.
\end{proof}

\begin{remark}The inequality \eqref{eqM13} can also be
obtained using Mirzakhani's recursion formula \eqref{eqM3}. Let
$f(x)=\zeta(2x)(1-2^{1-2x})$, we can check that $f''(x)<0$ when
$x\geq1$. This implies that $\{a_{n+1}-a_n\}_{n\geq1}$ is a
decreasing sequence. By Mirzakhani's recursion formula \eqref{eqM3},
we have
\begin{equation}
V_{g,n+1}-[\tau_1\tau_0^n]_{g,n+1}\leq\frac{a_1-a_0}{a_1}V_{g,n+1}.
\end{equation}
Substituting $a_0=\frac{1}{2}$ and $a_1=\frac{\pi^2}{12}$, we get
$$[\tau_1\tau_0^n]_{g,n+1}\geq\frac{6}{\pi^2}V_{g,n+1}.$$

\end{remark}

\begin{corollary}
For any $g,n\geq0$, we have
\begin{equation} \label{eqM12}
V_{g,n+1}>12(2g-2+n)V_{g,n}\quad \text{ and  }\quad
V_{g,n+1}<C(2g-2+n)V_{g,n},
\end{equation}
where $C=\frac{20\pi^2}{10-\pi^2}=1513.794\ldots$.

\end{corollary}

\begin{proof} It is not difficult to see that the coefficients in
\eqref{dilaton}
$$\left\{\frac12 (L+1) \frac{\pi^{2L}}{(2L+3)!}\right\}_{L\geq0}$$
is a decreasing sequence.

Taking all $d_i=0$ in \eqref{dilaton} and keeping only the first
term in the right-hand side, we get

$$(2g-2+n)V_{g,n}\leq\frac{1}{12}[\tau_1\tau_0^n]_{g,n+1}<\frac{1}{12}V_{g,n+1},$$
which is the first inequality in \eqref{eqM12}.

If we take first two terms in the right-hand side of \eqref{dilaton}
and apply Lemma \ref{M11}, we get
\begin{align*}
(2g-2+n)V_{g,n}&\geq\frac{1}{12}[\tau_1\tau_0^n]_{g,n+1}-\frac{\pi^2}{120}[\tau_2\tau_0^n]_{g,n+1}\\
&>(\frac{1}{12}-\frac{\pi^2}{120})[\tau_1\tau_0^n]_{g,n+1}\\
&\geq\frac{10-\pi^2}{120}\cdot\frac{6}{\pi^2}V_{g,n+1}\\
&=\frac{10-\pi^2}{20\pi^2}V_{g,n+1},
\end{align*}
which is the second inequality in \eqref{eqM12}.
\end{proof}

The inequalities \eqref{eqM12} imply that
$$12\leq\liminf_{g\rightarrow\infty}\frac{V_{g,n(g)+1}}{(2g-2+n(g))V_{g,n(g)}}\leq
\limsup_{g\rightarrow\infty}\frac{V_{g,n(g)+1}}{(2g-2+n(g))V_{g,n(g)}}\leq
\frac{20\pi^2}{10-\pi^2},$$ where $n(g) \rightarrow \infty$ as $g
\rightarrow \infty$.

$$ \ \ \ \ $$

\end{document}